\pdfoutput=1
\RequirePackage{ifpdf}
\ifpdf 
\documentclass[pdftex]{sigma}
\else
\documentclass{sigma}
\fi

\usepackage{bm}

\def\im{\mathrm{i}}
\def\d{\mathrm{d}}
\def\e{\mathrm{e}}
\def\erfc{\operatorname{erfc}}
\def\erf{\operatorname{erf}}

\numberwithin{equation}{section}

\newtheorem{Theorem}{Theorem}[section]
\newtheorem*{Theorem*}{Theorem}

\newtheorem{Lemma}[Theorem]{Lemma}
\newtheorem{Proposition}[Theorem]{Proposition}
 { \theoremstyle{definition}

\newtheorem{Remark}[Theorem]{Remark} }

\begin{document}
\allowdisplaybreaks

\renewcommand{\thefootnote}{}

\newcommand{\arXivNumber}{2401.16671}

\renewcommand{\PaperNumber}{026}

\FirstPageHeading

\ShortArticleName{Resurgence in the Transition Region: The Incomplete Gamma Function}

\ArticleName{Resurgence in the Transition Region:\\ The Incomplete Gamma Function\footnote{This paper is a~contribution to the Special Issue on Asymptotics and Applications of Special Functions in Memory of Richard Paris. The~full collection is available at \href{https://www.emis.de/journals/SIGMA/Paris.html}{https://www.emis.de/journals/SIGMA/Paris.html}}}

\Author{Gerg\H{o} NEMES}

\AuthorNameForHeading{G.~Nemes}

\Address{Department of Physics, Tokyo Metropolitan University,\\ 1--1 Minami-osawa, Hachioji-shi, Tokyo, 192-0397, Japan}
\Email{\href{mailto:nemes@tmu.ac.jp}{nemes@tmu.ac.jp}}

\ArticleDates{Received January 31, 2024, in final form March 24, 2024; Published online March 31, 2024}

\Abstract{We study the resurgence properties of the coefficients $C_n(\tau)$ appearing in the asymptotic expansion of the incomplete gamma function within the transition region. Our findings reveal that the asymptotic behaviour of $C_n(\tau)$ as $n\to +\infty$ depends on the parity of $n$. Both $C_{2n-1}(\tau)$ and $C_{2n}(\tau)$ exhibit behaviours characterised by a leading term accompanied by an inverse factorial series, where the coefficients are once again $C_{2k-1}(\tau)$ and $C_{2k}(\tau)$, respectively. Our derivation employs elementary tools and relies on the known resurgence properties of the asymptotic expansion of the gamma function and the uniform asymptotic expansion of the incomplete gamma function. To the best of our knowledge, prior to this paper, there has been no investigation in the existing literature regarding the resurgence properties of asymptotic expansions in transition regions.}


\Classification{34E05; 33B20}

\renewcommand{\thefootnote}{\arabic{footnote}}
\setcounter{footnote}{0}

\section{Introduction and main result}

Resurgence refers to the phenomenon wherein the ``late'' coefficients and remainder terms in an asymptotic expansion can be re-expanded as generalised asymptotic expansions, with their coefficients corresponding to the ``early" coefficients in the original expansion. Initially observed by Dingle \cite{Dingle1973} and further developed by \'Ecalle \cite{Ecalle1981,Ecalle1981-II,Ecalle1981-III} in his theory of resurgent functions, this phenomenon has been subsequently identified in a broad class of asymptotic expansions. Examples include asymptotic expansions for integrals with saddles \cite{Bennett2018,Berry1991,Howls1997}, integrals with coalescing saddles \cite{OldeDaalhuis2000}, ordinary differential equations with irregular singularities at infinity \cite{Murphy1997, OldeDaalhuis1998b}, nonlinear ordinary differential equations \cite{OldeDaalhuis2005a,OldeDaalhuis2005b}, and second-order difference equations \cite{OldeDaalhuis2004}.

In cases where a problem involves additional parameters alongside the asymptotic variable, the coefficients of the asymptotic expansion can become singular as these parameters approach specific critical values, termed transition points. From a numerical perspective, the asymptotic expansion becomes impractical in an entire region surrounding a transition point. This region is known as a transition region, and its extent depends on the relative magnitudes of the corresponding parameter and the asymptotic variable, as well as the nature of the transition point. To address the computational challenge posed by transition regions, a commonly employed approach is the utilisation of uniform asymptotic expansions. Uniform asymptotic expansions are valid in large domains, including the transition region. However, the coefficients of such expansions are complicated functions with removable singularities, making them difficult to compute. Another solution is offered by the lesser-known transitional asymptotic expansions, providing an alternative approach. In contrast to their uniform counterparts, these expansions hold validity within smaller domains. However, they possess simple polynomial coefficients that are easy to compute. The region of validity for a transitional expansion is large enough to cover the transition region, making it a useful alternative to uniform expansions.

As far as we are aware, there has been no investigation in the current literature concerning the resurgence properties of transitional asymptotic expansions. This paper marks the beginning of research in this direction, with a specific emphasis on the transitional asymptotic expansion for the (normalised) incomplete gamma function $Q(a, z)$ \cite[Section~8.2\,(i)]{DLMF}. In this case, the role of the large asymptotic variable is played by $a$, while the relevant parameter in the problem is $\lambda=z/a$. The transition region, surrounding the transition point at $\lambda=1$, has a width of~$\mathcal{O}\big(|a|^{-1/2}\big)$, and within this region, the asymptotic behaviour of~$Q(a,z)$ undergoes an abrupt change. Specifically, $Q(a,z)$ demonstrates a sharp decay near the transition point $\lambda=1$, as it approaches unity for $\lambda<1$ and progressively decreases algebraically to zero for $\lambda>1$.

In the paper \cite{Nemes2019}, it was established that the incomplete gamma function admits the transitional asymptotic expansion
\begin{equation}\label{eq2}
Q\big(a,a+\tau a^{\frac{1}{2}}\big) \sim \frac{1}{2} \erfc\big(2^{-\frac{1}{2}}\tau\big) + \frac{1}{\sqrt {2\pi a}}
\exp\left(-\frac{\tau^2}{2}\right) \sum_{n=0}^\infty\frac{C_n(\tau)}{a^{n/2}}
\end{equation}
as $a\to \infty$ in the sector $|\arg a| \le \pi-\delta<\pi$, uniformly with respect to bounded complex values of~$\tau$. Here, $\erfc$ denotes the complementary error function \cite[equation~(2.2)]{DLMF}. The coefficients~$C_n(\tau)$ are polynomials in $\tau$ of degree $3n+2$ and satisfy
\begin{gather}
C_0 (\tau ) = \frac{1}{3}\tau^2-\frac{1}{3},\nonumber\\
C_n (\tau ) + \tau C'_n (\tau ) - C''_n (\tau ) = \tau \big(\tau ^2 - 2\big)C_{n - 1} (\tau ) - \big(2\tau^2 - 1\big)C'_{n - 1} (\tau ) + \tau C''_{n - 1} (\tau )\label{eq32}
\end{gather}
for $n\geq 1$. Additionally, the even- and odd-order polynomials are even and odd functions, respectively. An algorithm for generating these coefficients is discussed in Appendix \ref{appendixb}.

In this paper, we are interested in the asymptotic behaviour of the coefficients $C_n(\tau)$ as $n\to+\infty$. Our main result can be summarised as follows.

\begin{Theorem}\label{thm1} The coefficients $C_n(\tau)$ posses the inverse factorial series
\begin{align}
& C_{2n - 1} (\tau ) \sim \frac{\Gamma (n)}{(2\pi )^{n + 1/2}}\sin \left( \frac{n }{2}\pi \right)\exp \left( \frac{\tau ^2 }{2} \right)\erf\big( 2^{ - \frac{1}{2}} \tau \big) \nonumber\\
&\phantom{C_{2n - 1} (\tau ) \sim}{}- \frac{1}{\pi }\sum_{k = 1}^\infty C_{2k - 1} (\tau )\sin \left( \frac{n - k}{2}\pi \right)\frac{\Gamma (n - k)}{(2\pi )^{n - k} } ,\label{eq14}
\\ &
C_{2n} (\tau ) \sim - \frac{\Gamma \left( n + \frac{1}{2} \right)}{(2\pi )^{n + 1} }\sin \left( \left( n + \frac{1}{2} \right)\frac{\pi }{2} \right)\exp \left( \frac{\tau ^2}{2} \right) \nonumber\\
&\phantom{C_{2n} (\tau ) \sim}{}- \frac{1}{\pi }\sum\limits_{k = 0}^\infty C_{2k} (\tau )\sin \left( \frac{n - k}{2}\pi \right)\frac{\Gamma (n - k)}{(2\pi )^{n - k} }\label{eq15}
\end{align}
as $n\to +\infty$, uniformly with respect to bounded complex values of $\tau$. Here, $\erf$ denotes the error function {\rm\cite[equation~(7.2.1)]{DLMF}}. Moreover, under the condition that $\tau = o\big(n^{1/6}\big)$, these expansions hold as generalised asymptotic expansions {\rm\cite[Section~2.1\,(v)]{DLMF}} when $n\to +\infty$.
\end{Theorem}

Hence, the coefficients $C_n(\tau)$ exhibit a resurgence property. Their asymptotic behaviour expressed as an inverse factorial series, a characteristic feature in resurgence theory, with the coefficients in these series being once again the coefficients $C_k(\tau)$. As observed, the form of the inverse factorial series for $C_n(\tau)$ depends on the parity of $n$. Specifically, the expansion coefficients for odd $n$ comprise the $C_k(\tau)$ with odd $k$, while those for even $n$ involve the $C_k(\tau)$ with even $k$. Additionally, it is worth noting that the leading term is absent in the expansion of~$C_n(\tau)$ when $n \equiv 3 \bmod{4}$.

While the structure of the transitional expansion mirrors that of the corresponding uniform expansion (compare \eqref{eq2} and \eqref{eq23}), the resurgence properties of the former are notably simpler than those of the latter. The asymptotics of the coefficients in the uniform asymptotic expansion of $Q(a,z)$ are described in terms of the incomplete beta function, which is more intricate than the error function appearing in the expansion of $C_{2n - 1}(\tau)$. We suspect that this phenomenon holds more generally, such as when considering integrals with coalescing saddles. In such cases, the asymptotic behaviour of the uniform expansion coefficients is no longer represented by inverse factorial series but involves generalised asymptotic expansions with complicated functions. Therefore, from a resurgent perspective, transitional expansions seem to exhibit a much closer resemblance to simple (non-uniform) asymptotic expansions than to uniform ones.

A primary tool for deriving asymptotic expansions of coefficients in asymptotic series is the Borel transform. In the context of simple (non-uniform) asymptotic expansions, the Borel transform is a well-established technique that yields analytic functions with simple branch points. The asymptotic behaviour of the coefficients is subsequently obtained through the application of Darboux's method. This paper takes a different approach, employing elementary tools and drawing on the well-established resurgence properties of the asymptotic expansion of the gamma function, as well as the uniform asymptotic expansion of the incomplete gamma function. Details on these expansions are given in Appendix \ref{appendixa}.

The remaining part of the paper is structured as follows. In Section \ref{section2}, we present the proof of the main result. Section \ref{section3} offers numerical examples that demonstrate the applicability of the results. The paper concludes with a brief discussion in Section \ref{section4}.

\section{Proof of the main result}\label{section2}

The essence of proving our main result lies in relating the coefficients $C_n(\tau)$ to some coefficients~$D_n(\tau)$, whose asymptotic behaviour is easier to analyse using the known resurgence properties of the asymptotic expansion of the gamma function.

In the following discussion, we will assume that $|\tau| < |a|^{1/2}$ and $|\arg a| < \pi$. Utilising the definitions of the incomplete gamma function $Q(a,z)$ and the scaled gamma function $\Gamma^\ast (a)$ (refer to \eqref{eq24}), we can readily deduce that
\begin{align}
\frac{\partial Q\big(a,a + \tau a^{\frac{1}{2}}\big)}{\partial \tau} & = - \frac{a^{\frac{1}{2}} \e^{ - a}}{\Gamma (a)}\big( a + \tau a^{\frac{1}{2}} \big)^{a - 1} \e^{ - \tau a^{\frac{1}{2}} } = - \frac{1}{\sqrt{2\pi }}\frac{1}{\Gamma^\ast (a)}\big( 1 + \tau a^{ - \frac{1}{2}}\big)^{a - 1} \e^{ - \tau a^{\frac{1}{2}} }\nonumber \\ & = - \frac{1}{\sqrt {2\pi }}\exp\left( - \frac{\tau^2}{2} \right)\frac{1}{\Gamma^\ast (a)}\e^{F(a,\tau )} ,\label{eq1}
\end{align}
where
\[
F(a,\tau ) = \frac{\tau ^2}{2} + (a - 1)\log \big(1 + \tau a^{-\frac{1}{2}}\big) - \tau a^{\frac{1}{2}} .
\]
Using the Maclaurin series of the logarithm, we derive the expansion
\[
F(a,\tau ) = \sum_{n = 1}^\infty \left( \frac{1}{n}-\frac{\tau^2}{n + 2}\right)\frac{( - \tau )^n}{a^{n/2}} .
\]
By exponentiating both sides and expanding the right-hand side in negative powers of $a^{-1/2}$, we~obtain
\[
\e^{F(a,\tau )} = \sum_{n = 0}^\infty \frac{p_n (\tau )}{a^{n/2} },
\]
where $p_n(\tau)$ is a polynomial in $\tau$ of degree $3n$. By substituting this expansion into \eqref{eq1} and employing the asymptotic series \eqref{eq25} of the reciprocal scaled gamma function, we deduce the asymptotic expansion
\[
\frac{\partial Q\big(a,a + \tau a^{\frac{1}{2}}\big)}{\partial \tau} \sim - \frac{1}{\sqrt{2\pi}}\exp \left(- \frac{\tau^2}{2}\right)\left( 1 + \sum_{n = 0}^\infty \frac{D_n (\tau )}{a^{(n + 1)/2}} \right)
\]
as $a\to \infty$ in the sector $|\arg a| \le \pi-\delta<\pi$, uniformly with respect to bounded complex values of~$\tau$. The coefficients $D_n(\tau)$ are expressed in terms of the Stirling coefficients $\gamma_n$ (refer to Appendix~\ref{appendixa}) and the polynomials $p_n(\tau)$ as follows:
\begin{equation}\label{eq4}
D_{2n - 1} (\tau ) = \sum_{k = 0}^n \gamma _{n - k} p_{2k} (\tau ) ,\qquad D_{2n} (\tau ) = \sum_{k = 0}^n \gamma _{n - k} p_{2k + 1} (\tau ).
\end{equation}
We note that $D_n(\tau)$ is a polynomial in $\tau$ of degree $3n+3$. On the other hand, differentiating each side of \eqref{eq2} with respect to $\tau$ yields the asymptotic expansion
\[
\frac{\partial Q\big(a,a + \tau a^{\frac{1}{2}}\big)}{\partial \tau} \sim - \frac{1}{\sqrt {2\pi }}\exp \left( - \frac{\tau^2}{2}\right)\left( 1 + \sum_{n = 0}^\infty \frac{\tau C_n (\tau ) - C'_n (\tau )}{a^{(n + 1)/2}} \right)
\]
as $a\to \infty$ in the sector $|\arg a| \le \pi-\delta<\pi$, uniformly with respect to bounded complex values of~$\tau$. Therefore, by invoking the uniqueness property of the coefficients in an asymptotic expansion, we deduce the following relation between the coefficients $C_n(\tau)$ and $D_n(\tau)$: $\tau C_n (\tau ) - C'_n (\tau ) = D_n (\tau )$. Solving for $C_n(\tau)$ gives the following expressions:
\begin{equation}\label{eq7}
C_{2n-1} (\tau ) = -\exp \left(\frac{\tau^2}{2} \right)\int_0^\tau \exp \left(- \frac{t^2 }{2}\right)D_{2n-1} (t)\d t,
\end{equation}
and
\begin{equation}\label{eq8}
C_{2n} (\tau ) = \exp \left(\frac{\tau^2}{2}\right) C_{2n} (0) - \exp \left( \frac{\tau^2}{2}\right)\int_0^\tau \exp \left(-\frac{t^2}{2}\right)D_{2n} (t)\d t.
\end{equation}
Here, we utilised the fact that $C_{2n-1} (\tau )$, being an odd function, satisfies $C_{2n-1} (0)=0$.

In the following proposition, we present detailed information regarding the asymptotic behaviour of $D_n(\tau)$ as $n$ becomes large, expressed as a truncated inverse factorial series along with a remainder. This result, in combination with \eqref{eq7} and \eqref{eq8}, will be used to establish our main result. In this paper, empty sums are interpreted as zero, and subscripts in the $\mathcal{O}$ notations indicate the dependence of the implied constant on certain parameters.

\begin{Proposition}\label{prop1} For any non-negative integer $N$, the coefficients $D_n(\tau)$ exhibit the following asymptotic behaviours:
\begin{align}
D_{2n - 1} (\tau ) ={}& -\frac{1}{\pi}\frac{\Gamma (n)}{(2\pi )^n } \sin \left( \frac{n}{2}\pi \right) -\frac{1}{\pi}\sum_{k = 1}^{N - 1} D_{2k - 1} (\tau )\sin \left( \frac{n - k}{2}\pi \right)\frac{\Gamma (n - k)}{(2\pi )^{n - k}}\nonumber \\ & +\mathcal{O}_N (1)\bigl(\left| \tau \right|^3 + 1\bigr)^{2N} \exp \big( 2\pi \bigl(\left| \tau \right|^3 + 1\bigr)^2 \big)\frac{\Gamma (n - N)}{(2\pi )^{n - N}}\label{eq10}
\end{align}
and
\begin{align}
D_{2n} (\tau ) ={}& -\frac{1}{\pi}\sum_{k = 0}^{N - 1} D_{2k} (\tau )\sin \left( \frac{n - k}{2}\pi \right)\frac{\Gamma (n - k)}{(2\pi )^{n - k}} \nonumber\\ & +\mathcal{O}_N (1)\bigl(\left| \tau \right|^3 + 1\bigr)^{2N+1} \exp \big( 2\pi \bigl(\left| \tau \right|^3 + 1\bigr)^2 \big)\frac{\Gamma (n - N)}{(2\pi )^{n - N}}\label{eq11}
\end{align}
as $n\to+\infty$, provided $\tau = o\big(n^{1/6}\big)$.
\end{Proposition}

In order to prove Proposition \ref{prop1}, we need to establish some lemmas.

\begin{Lemma}\label{lemma1} For every non-negative integer $k$ and any complex number $\tau$, the inequality
\[
\left| p_k (\tau )\right| \le 4 \bigl(\left| \tau \right|^3 + 1\bigr)^k
\]
is satisfied.
\end{Lemma}

\begin{proof}
Starting from the expansion defining the polynomials $p_k(\tau)$ and using Cauchy's formula, we find that
\begin{equation}\label{eq3}
p_k (\tau ) = \frac{1}{2\pi \im}\oint_{\left| t \right| = r} \exp \bigg( \sum_{n = 1}^\infty \left( \frac{1}{n} - \frac{\tau ^2 }{n + 2} \right)( - \tau t)^n \bigg)\frac{\d t}{t^{k + 1}} ,
\end{equation}
where $r = \frac{1}{\left| \tau \right|^3 + 1}$. Assuming $\left|t\right| = \frac{1}{\left| \tau \right|^3 + 1}$, we can bound the exponential part as follows:
\begin{align*}
\left| \exp \bigg( \sum_{n = 1}^\infty \left( \frac{1}{n} - \frac{\tau ^2}{n + 2} \right)( - \tau t)^n \bigg) \right| & \le \exp \bigg( \bigg( \frac{\left| \tau \right|^2 }{3} + 1 \bigg)\sum_{n = 1}^\infty \left| \tau \right|^n \left| t \right|^n \bigg) \\ & = \exp \bigg( \bigg( \frac{\left| \tau \right|^2 }{3} + 1 \bigg)\sum_{n = 1}^\infty \bigg( \frac{\left| \tau \right|}{\left| \tau \right|^3 + 1} \bigg)^n \bigg) \\ & = \exp \bigg( \bigg( \frac{\left| \tau \right|^2 }{3} + 1 \bigg)\frac{\left| \tau \right|}{\left| \tau \right|^3 + 1 - \left| \tau \right|} \bigg) <4.
\end{align*}
Applying this bound in \eqref{eq3} and performing a straightforward estimation yields the desired result.
\end{proof}

\begin{Lemma}\label{lemma2} For every positive integer $n$, the inequality
\[
\left|\gamma_n\right| \le \frac{\Gamma(n)}{(2\pi)^n}
\]
is satisfied.
\end{Lemma}

\begin{proof} For the case $n=1$, the inequality holds trivially since
\[
\left| \gamma _1 \right| = \frac{1}{12} < \frac{1}{2\pi} = \frac{\Gamma (1)}{2\pi}.
\]
Now, assuming $n\ge 2$, let $R_n$ denote the error incurred by truncating the asymptotic expansion of the scaled gamma function after $n$ terms:
\[
\Gamma^\ast (a)=\sum_{k = 0}^{n - 1} (-1)^k \frac{\gamma _k}{a^k} + R_n (a).
\]
In the paper \cite[equation~(3.11)]{Boyd1994}, it was established that for $n\ge 2$ and $a>0$, the following inequality holds:
\[
\left| a^n R_n (a) \right| \le \frac{1 + \zeta (n)}{2\pi }\frac{\Gamma (n)}{(2\pi )^n}.
\]
Here, $\zeta$ denotes the Riemann zeta function \cite[equation~(25.2.1)]{DLMF}. Consequently,
\[
\left| \gamma_n \right| = \lim_{a \to + \infty } \left| a^n R_n (a) \right| \le \frac{1 + \zeta (n)}{2\pi }\frac{\Gamma (n)}{(2\pi )^n} \le \frac{\Gamma (n)}{(2\pi )^n}
\]
provided $n\ge 2$.
\end{proof}

\begin{proof}[Proof of Proposition \ref{prop1}.] We will prove the statement for the coefficients $D_{2n - 1} (\tau )$. The proof for $D_{2n} (\tau)$ follows in a completely analogous manner. Let $N$ be any non-negative integer, and express \eqref{eq4} in the following form:
\begin{equation}\label{eq20}
D_{2n - 1} (\tau ) = \sum_{k = 0}^{N - 1}\gamma _{n - k} p_{2k} (\tau ) + \sum_{k = N}^n \gamma _{n - k} p_{2k} (\tau ).
\end{equation}
We shall establish an asymptotic order estimate for the second sum as $n\to+\infty$. To do so, we first observe that if $q$ and $r$ are non-negative integers, then $q!r! \le (q + r)!$. Thus,
\[
\Gamma (n - N - k)k! = (n - N - 1 - k)!k! \le (n - N - 1)! = \Gamma (n - N)
\]
for integers $n$, $N$, and $k$ satisfying $n \ge N+k+1$ and $k\ge 0$. Employing this inequality in combination with those provided in Lemmas \ref{lemma1} and \ref{lemma2}, we deduce that
\begin{align*}
& \left| \sum_{k = N}^n \gamma _{n - k} p_{2k} (\tau ) \right| \\
&\qquad= \left| \sum_{k = N}^{n-1} \gamma _{n - k} p_{2k} (\tau ) + p_{2n} (\tau )\right| \le 4\sum_{k = N}^{n - 1} \frac{\Gamma (n - k)}{(2\pi )^{n - k} }\bigl(\left| \tau \right|^3 + 1\bigr)^{2k} + 4\bigl(\left| \tau \right|^3 + 1\bigr)^{2n}
\\ & \qquad= 4\frac{\Gamma (n - N)}{(2\pi )^n }\left( \sum_{k = N}^{n - 1} \frac{\Gamma (n - k)}{\Gamma(n - N)}\big(2\pi \bigl(\left| \tau \right|^3 + 1\bigr)^2 \big)^k + \frac{\big(2\pi \bigl(\left| \tau \right|^3 + 1\bigr)^2 \big)^n }{\Gamma (n - N)} \right)
\\ & \qquad= 4\big(2\pi \bigl(\left| \tau \right|^3 + 1\bigr)^2 \big)^N \frac{\Gamma(n - N)}{(2\pi )^n }\\
&\phantom{\qquad=}{}\times\left( \sum_{k = 0}^{n - N - 1} \frac{\Gamma (n - N - k)}{\Gamma (n - N)}\big(2\pi \bigl(\left| \tau \right|^3 + 1\bigr)^2 \big)^k + \frac{\big(2\pi \bigl(\left| \tau \right|^3 + 1\bigr)^2 \big)^{n - N} }{\Gamma(n - N)} \right)
\\ & \qquad\le 4\big(2\pi \bigl(\left| \tau \right|^3 + 1\bigr)^2 \big)^N \frac{\Gamma(n - N)}{(2\pi )^n }\left( \sum_{k = 0}^{n - N - 1} \frac{\big(2\pi \bigl(\left| \tau \right|^3 + 1\bigr)^2 \big)^k }{k!} + \frac{\big(2\pi \bigl(\left| \tau \right|^3 + 1\bigr)^2 \big)^{n - N} }{\Gamma(n - N)}\right)
\\ & \qquad\le 4\big(2\pi \bigl(\left| \tau \right|^3 + 1\bigr)^2 \big)^N \frac{\Gamma(n - N)}{(2\pi )^n }\left( \exp \big( 2\pi \bigl(\left| \tau \right|^3 + 1\bigr)^2 \big) + \frac{\big(2\pi \bigl(\left| \tau \right|^3 + 1\bigr)^2 \big)^{n - N} }{\Gamma(n - N)} \right)
\\ &\qquad = \mathcal{O}_N (1)\bigl(\left| \tau \right|^3 + 1\bigr)^{2N} \exp \big( 2\pi \bigl(\left| \tau \right|^3 + 1\bigr)^2 \big)\frac{\Gamma(n - N)}{(2\pi )^{n - N} }
\end{align*}
as $n\to+\infty$, provided $\tau = o\big(n^{1/6}\big)$. The condition $\tau = o\big(n^{1/6}\big)$ on $\tau$ is imposed to ensure that
\[
\frac{\big(2\pi \bigl(\left| \tau \right|^3 + 1\bigr)^2 \big)^{n - N} }{\Gamma(n - N)} = \mathcal{O}_N (1) \exp \big( 2\pi \bigl(\left| \tau \right|^3 + 1\bigr)^2 \big).
\]
We have thus established that, for any fixed non-negative integer $N$,
\begin{equation}\label{eq5}
D_{2n - 1} (\tau ) = \sum_{k = 0}^{N - 1}\gamma _{n - k} p_{2k} (\tau ) + \mathcal{O}_N (1)\bigl(\left| \tau \right|^3 + 1\bigr)^{2N} \exp \big( 2\pi \bigl(\left| \tau \right|^3 + 1\bigr)^2 \big)\frac{\Gamma(n - N)}{(2\pi )^{n - N} }
\end{equation}
as $n\to+\infty$, provided $\tau = o\big(n^{1/6}\big)$.

Continuing, our aim now is to derive a precise asymptotic expression for the sum in \eqref{eq5} for large $n$, using the known inverse factorial series of the Stirling coefficients. Referring to \eqref{eq18}, we can assert that for each fixed non-negative integer $N$:
\[
\gamma _n = -\frac{1}{\pi}\sum_{m = 0}^{N- 1} \gamma _m \sin\left( \frac{n - m}{2}\pi \right)\frac{\Gamma (n - m)}{(2\pi )^{n - m}} + \mathcal{O}_N(1) \frac{\Gamma (n - N)}{(2\pi )^{n - N} }
\]
as $n\to+\infty$. Applying this expression to the sum in \eqref{eq5}, we infer that
\begin{align*}
\sum_{k = 0}^{N - 1} \gamma _{n - k} p_{2k} (\tau ) = & -\sum_{k = 0}^{N - 1} \frac{1}{\pi}\sum_{m = 0}^{N - 1} \gamma _m \sin \left( \frac{n - k - m}{2}\pi \right)\frac{\Gamma (n - k - m)}{(2\pi )^{n - k - m} } p_{2k} (\tau ) \\ & + \sum_{k = 0}^{N - 1} \mathcal{O}_N(1) \frac{\Gamma (n - k - N)}{(2\pi )^{n - k - N} }p_{2k} (\tau )
\\ = &-\frac{1}{\pi }\sum_{k = 0}^{N - 1} \left( \sum_{j = 0}^k \gamma _{k - j} p_{2j} (\tau ) \right)\sin \left( \frac{n - k}{2}\pi \right)\frac{\Gamma (n - k)}{(2\pi )^{n - k} } \\ &- \frac{1}{\pi }\sum_{k = N}^{2N- 2} \left( \sum_{j = k - N + 1}^{N-1} \gamma _{k - j} p_{2j} (\tau ) \right)\sin \left( \frac{n - k}{2}\pi\right)\frac{\Gamma (n - k)}{(2\pi )^{n - k}}
\\ & + \sum_{k = 0}^{N - 1} \mathcal{O}_N(1) \frac{\Gamma (n - k - N)}{(2\pi )^{n - k - N} }p_{2k} (\tau )
\end{align*}
as $n\to+\infty$, with any fixed non-negative integer $N$. From \eqref{eq4}, we can affirm that
\begin{align*}
&-\frac{1}{\pi }\sum_{k = 0}^{N - 1} \left( \sum_{j = 0}^k \gamma _{k - j} p_{2j} (\tau ) \right)\sin \left( \frac{n - k}{2}\pi \right)\frac{\Gamma (n - k)}{(2\pi )^{n - k} } \\ &\qquad= -\frac{1}{\pi}\frac{\Gamma (n)}{(2\pi )^n } \sin \left( \frac{n}{2}\pi \right)-\frac{1}{\pi}\sum_{k = 1}^{N - 1} D_{2k - 1} (\tau )\sin \left( \frac{n - k}{2}\pi \right)\frac{\Gamma (n - k)}{(2\pi )^{n - k}}.
\end{align*}
Moreover, it is evident that, for any fixed non-negative integer $N$,
\[
- \frac{1}{\pi }\sum_{k = N}^{2N-2} \left( \sum_{j = k - N + 1}^{N-1} \gamma _{k - j} p_{2j} (\tau ) \right)\sin \left( \frac{n - k}{2}\pi\right)\frac{\Gamma (n - k)}{(2\pi )^{n - k}}= \mathcal{O}_N (1)\bigl(\left| \tau \right|^3 + 1\bigr)^{2N} \frac{\Gamma (n - N)}{(2\pi )^{n - N}}
\]
and
\[
\sum_{k = 0}^{N - 1} \mathcal{O}_N(1) \frac{\Gamma (n - k - N)}{(2\pi )^{n - k - N} }p_{2k} (\tau )=\mathcal{O}_N (1)\bigl(\left| \tau \right|^3 + 1\bigr)^{2N} \frac{\Gamma (n - N)}{(2\pi )^{n - N}}
\]
as $n\to+\infty$, uniformly with respect to $\tau$. Consequently, in summary,
\begin{align*}
\sum_{k = 0}^{N - 1} \gamma _{n - k} p_{2k} (\tau ) ={}& -\frac{1}{\pi}\frac{\Gamma (n)}{(2\pi )^n } \sin \left( \frac{n}{2}\pi \right) -\frac{1}{\pi}\sum_{k = 1}^{N - 1} D_{2k - 1} (\tau )\sin \left( \frac{n - k}{2}\pi \right)\frac{\Gamma (n - k)}{(2\pi )^{n - k}} \\ & +\mathcal{O}_N (1)\bigl(\left| \tau \right|^3 + 1\bigr)^{2N} \frac{\Gamma (n - N)}{(2\pi )^{n - N}}
\end{align*}
as $n\to+\infty$, with any fixed non-negative integer $N$ and uniformly with respect to $\tau$. Inserting this asymptotic expression into \eqref{eq5} yields the desired result for $D_{2n - 1} (\tau )$.
\end{proof}

We are now in a position to prove Theorem \ref{thm1}.

\begin{proof}[Proof of Theorem \ref{thm1}.] By employing \eqref{eq10} in \eqref{eq7} with $t$ in place of $\tau$, integrating term-by-term, and utilising \eqref{eq7} for each term, we deduce
\begin{align*}
C_{2n - 1} (\tau ) ={}& \frac{\Gamma (n)}{(2\pi )^{n + 1/2} }\sin \left( \frac{n}{2}\pi \right)\exp \left( \frac{\tau ^2}{2} \right)\erf\big( 2^{ - \frac{1}{2}} \tau \big) \\
&- \frac{1}{\pi}\sum_{k = 1}^{N - 1} C_{2k - 1} (\tau )\sin \left( \frac{n - k}{2}\pi \right)\frac{\Gamma (n - k)}{(2\pi )^{n - k} }
\\ & - \exp \left( \frac{\tau ^2}{2} \right)\int_0^\tau \mathcal{O}_N (1)\exp \left( - \frac{t^2}{2} \right)\bigl(\left| t \right|^3 + 1\bigr)^{2N} \\
&\phantom{-}{}\times\exp \big( 2\pi \bigl(\left| t \right|^3 + 1\bigr)^2 \big)\d t \frac{\Gamma (n - N)}{(2\pi )^{n - N}}
\end{align*}
as $n\to+\infty$, provided that $N\ge 0$ is fixed and $\tau = o\big(n^{1/6}\big)$. Choosing the integration path in the error term to be a straight line segment connecting $0$ to $\tau$ and employing a straightforward estimation, we obtain the following bound:
\begin{align*}
& \left| \exp \left( \frac{\tau ^2}{2} \right)\int_0^\tau \mathcal{O}_N (1)\exp \left( - \frac{t^2}{2} \right)\bigl(\left| t \right|^3 + 1\bigr)^{2N} \exp \big( 2\pi \bigl(\left| t \right|^3 + 1\bigr)^2 \big)\d t \right|
\\ &\qquad \le \mathcal{O}_N (1)\bigl(\left| \tau \right|^3 + 1\bigr)^{2N} \exp \big( 2\pi \bigl(\left| \tau \right|^3 + 1\bigr)^2 \big) \int_0^{\left| \tau \right|} \left| \exp \left( \frac{\tau ^2 - t^2}{2} \right) \right|\d\left| t \right|
\\ & \qquad= \mathcal{O}_N (1) \left| \tau \right|\bigl(\left| \tau \right|^3 + 1\bigr)^{2N} \exp \big( 2\pi \bigl(\left| \tau \right|^3 + 1\bigr)^2 \big) \int_0^1 \exp \bigg( \frac{\left| \tau \right|^2}{2}(1 - s^2 )\cos (2\arg \tau ) \bigg)\d s
\\ & \qquad= \mathcal{O}_N (1)\left| \tau \right|\bigl(\left| \tau \right|^3 + 1\bigr)^{2N} \exp \big( 2\pi \bigl(\left| \tau \right|^3 + 1\bigr)^2 \big)\exp \bigg( \frac{\left| \tau \right|^2}{2} \bigg)
\\ & \qquad= \mathcal{O}_N (1)\left| \tau \right|\bigl(\left| \tau \right|^3 + 1\bigr)^{2N} \exp \big( 7\bigl(\left| \tau \right|^3 + 1\bigr)^2 \big).
\end{align*}
Accordingly,
\begin{align}
C_{2n - 1} (\tau ) ={}& \frac{\Gamma (n)}{(2\pi )^{n + 1/2} }\sin \left( \frac{n}{2}\pi \right)\exp \left( \frac{\tau ^2}{2} \right)\erf\big( 2^{ - \frac{1}{2}} \tau \big)\nonumber\\
& - \frac{1}{\pi}\sum_{k = 1}^{N - 1} C_{2k - 1} (\tau )\sin \left( \frac{n - k}{2}\pi \right)\frac{\Gamma (n - k)}{(2\pi )^{n - k} }\nonumber
\\ & +\mathcal{O}_N (1)\left| \tau \right|\bigl(\left| \tau \right|^3 + 1\bigr)^{2N} \exp \big( 7\bigl(\left| \tau \right|^3 + 1\bigr)^2 \big) \frac{\Gamma (n - N)}{(2\pi )^{n - N}}\label{eq12}
\end{align}
as $n\to+\infty$, under the condition that $N\ge 0$ is fixed and $\tau = o\big(n^{1/6}\big)$. In a similar fashion, by employing \eqref{eq8} and \eqref{eq11}, we find that for each fixed non-negative integer $N$:
\begin{align*}
C_{2n} (\tau ) ={}& \exp \left( \frac{\tau ^2}{2} \right)C_{2n} (0) - \frac{1}{\pi}\sum_{k = 0}^{N - 1} \left( C_{2k} (\tau ) - \exp \left( \frac{\tau ^2}{2} \right)C_{2k} (0) \right)\sin \left( \frac{n - k}{2}\pi \right)
\\ &\times\frac{\Gamma (n - k)}{(2\pi )^{n - k}} + \mathcal{O}_N (1)\left| \tau \right|\bigl(\left| \tau \right|^3 + 1\bigr)^{2N + 1} \exp \big( 7\bigl(\left| \tau \right|^3 + 1\bigr)^2 \big)\frac{\Gamma (n - N)}{(2\pi )^{n - N}}
\end{align*}
as $n\to+\infty$, provided $\tau = o\big(n^{1/6}\big)$. To simplify this expression, we can refer to the known inverse factorial series of the coefficients $c_n(\eta)$ that appear in the uniform asymptotic expansion of the incomplete gamma function $Q(a,z)$. Specifically, since $C_{2n}(0)=c_n(0)$, we can deduce from \eqref{eq19} that for each fixed non-negative integer $N$:
\begin{align*}
C_{2n} (0) ={}& - \frac{\Gamma \left( n + \frac{1}{2} \right)}{(2\pi )^{n + 1}}\sin \left( \left(n + \frac{1}{2}\right)\frac{\pi}{2} \right) - \frac{1}{\pi }\sum_{k = 0}^{N - 1} C_{2k} (0)\sin \left( \frac{n - k}{2}\pi \right)\frac{\Gamma (n - k)}{(2\pi )^{n - k}} \\ & + \mathcal{O}_N(1)\frac{\Gamma (n - N)}{(2\pi )^{n - N} }
\end{align*}
as $n\to+\infty$. As a result,
\begin{align}
C_{2n} (\tau ) ={}& -\frac{\Gamma \left( n + \frac{1}{2} \right)}{(2\pi )^{n + 1} }\sin \left( \left( n + \frac{1}{2} \right)\frac{\pi }{2} \right)\exp \left( \frac{\tau ^2}{2} \right) - \frac{1}{\pi }\sum\limits_{k = 0}^{N-1} C_{2k} (\tau )\sin \left( \frac{n - k}{2}\pi \right)\nonumber
\\ &\times\frac{\Gamma (n - k)}{(2\pi )^{n - k} } + \mathcal{O}_N (1)(\left| \tau \right|+1)\bigl(\left| \tau \right|^3 + 1\bigr)^{2N + 1} \exp \big( 7\bigl(\left| \tau \right|^3 + 1\bigr)^2 \big)\frac{\Gamma (n - N)}{(2\pi )^{n - N}}\label{eq13}
\end{align}
as $n\to+\infty$, provided that $N\ge 0$ is fixed and $\tau = o\big(n^{1/6}\big)$.

If $\tau$ is bounded, the error terms in \eqref{eq12} and \eqref{eq13} are of the same order of magnitude as the first neglected terms in the sums. Consequently, the inverse factorial series expansions \eqref{eq14} and \eqref{eq15} are valid. Under the weaker assumption that $\tau = o\big(n^{1/6}\big)$, it can readily be verified that the ratio of consecutive remainder terms asymptotically decays in magnitude as $n\to+\infty$. This confirms that the expansions \eqref{eq14} and \eqref{eq15} remain valid in the context of generalised asymptotic expansions.
\end{proof}

\begin{Remark} Starting from \eqref{eq8}, it is not difficult to show that
\begin{equation}\label{eq21}
C_{2n}(\tau ) = \exp\left(\frac{\tau ^2}{2}\right)\int_\tau^{+\infty} \exp \left(-\frac{t^2}{2}\right)D_{2n} (t)\d t
\end{equation}
for all $n\ge 0$. However, this formula cannot be employed to derive the inverse factorial series~\eqref{eq15} from \eqref{eq11} due to the divergence of the improper integrals involving the error term in \eqref{eq11} against \smash{$\exp\bigl(-\frac{t^2}{2}\bigr)$}. An alternative approach involves substituting \eqref{eq20} into \eqref{eq21} initially and then estimating the second sum. Yet, obtaining an estimate resulting in an error term comparable to that in \eqref{eq13} proves challenging. Consequently, it appears necessary to invoke the inverse factorial series of the coefficients in the uniform asymptotic expansion.
\end{Remark}

\section{Numerical examples}\label{section3}

In this section, we present numerical results that confirm the accuracy of the inverse factorial expansions given in Theorem \ref{thm1}. Truncating the series \eqref{eq14} and \eqref{eq15} after $N$ and $N+1$ terms, respectively, we consider the approximations
\begin{align}
& C_{2n - 1} (\tau ) \approx \frac{\Gamma (n)}{(2\pi )^{n + 1/2}}\sin \left( \frac{n }{2}\pi \right)\exp \left( \frac{\tau ^2 }{2} \right)\erf\big( 2^{ - \frac{1}{2}} \tau \big) \nonumber \\&\phantom{ C_{2n - 1} (\tau ) \approx}{} - \frac{1}{\pi }\sum_{k = 1}^N C_{2k - 1} (\tau )\sin \left( \frac{n - k}{2}\pi \right)\frac{\Gamma (n - k)}{(2\pi )^{n - k} } ,\label{eq26}
\\ &
C_{2n} (\tau ) \approx - \frac{\Gamma \left( n + \frac{1}{2} \right)}{(2\pi )^{n + 1} }\sin \left( \left( n + \frac{1}{2} \right)\frac{\pi }{2} \right)\exp \left( \frac{\tau ^2}{2} \right)\nonumber\\
&\phantom{C_{2n} (\tau ) \approx}{} - \frac{1}{\pi }\sum\limits_{k = 0}^N C_{2k} (\tau )\sin \left( \frac{n - k}{2}\pi \right)\frac{\Gamma (n - k)}{(2\pi )^{n - k} }.\label{eq27}
\end{align}
In Table \ref{table1}, we provide exact numerical values of $C_{2n-1}(\tau)$, along with the approximate values derived from \eqref{eq26}, and the corresponding errors for various combinations of $n$, $\tau$, and the truncation index $N$. Correspondingly, Table \ref{table2} displays exact numerical values of $C_{2n}(\tau)$, the corresponding approximations from \eqref{eq27}, and the associated errors for different values of $n$, $\tau$, and $N$. The coefficients $C_n(\tau)$ were computed using the algorithm outlined in Appendix~\ref{appendixb}. It is observed, particularly from Table~\ref{table2}, that these approximations are effective only when the magnitude of $\tau$ is much smaller than that of~$n$.

\begin{table*}[!ht]\centering
\begin{tabular}
[c]{ l r @{\,}c@{\,} l}\hline
 & \\ [-1.5ex]
 values of $n$, $\tau$ and $N$ & $n=50$, $\tau=\frac{1}{2}$, $N=10$ & & \\ [0.5ex]
 exact numerical value of $C_{2n-1}(\tau)\qquad$ & $-0.1605549419108870432185698$ & $\times$ & $10^{20}$ \\ [0.5ex]
 approximation \eqref{eq26} to $C_{2n-1}(\tau)$ & $-0.1605549417678948233999888$ & $\times$ & $10^{20}$ \\ [0.5ex]
 error & $-0.1429922198185809$ & $\times$ & $10^{11}$\\ [-1.5ex]
 & \\\hline
& \\ [-1.5ex]
 values of $n$, $\tau$ and $N$ & $n=50$, $\tau=\frac{1}{2}$, $N=20$ & & \\ [0.5ex]
 exact numerical value of $C_{2n-1}(\tau)$ & $-0.1605549419108870432185698$ & $\times$ & $10^{20}$ \\ [0.5ex]
 approximation \eqref{eq26} to $C_{2n-1}(\tau)$ & $-0.1605549419107926902148106$ & $\times$ & $10^{20}$ \\ [0.5ex]
 error & $-0.943530037592$ & $\times$ & $10^{7}$\\ [-1.5ex]
 & \\\hline
& \\ [-1.5ex]
 values of $n$, $\tau$ and $N$ & $n=100$, $\tau=1+\frac{3}{2}\im$, $N=20$ & & \\ [0.5ex]
 exact numerical value of $C_{2n-1}(\tau)$ & $0.6104582432674722845198873$ & $\times$ & $10^{75}$ \\ [0.1ex]
 & $-\im\, 0.3947300517906354580061330$ & $\times$ & $10^{75}$ \\ [0.5ex]
 approximation \eqref{eq26} to $C_{2n-1}(\tau)$ & $0.6104582432678493030809359$ & $\times$ & $10^{75}$ \\ [0.1ex]
 & $-\im\, 0.3947300517908665670802024$ & $\times$ & $10^{75}$ \\ [0.5ex]
 error & $-0.3770185610486$ & $\times$ & $10^{63}$\\ [0.1ex]
& $+\im\, 0.2311090740694$ & $\times$ & $10^{63}$ \\ [-1.5ex]
 & \\\hline
& \\ [-1.5ex]
 values of $n$, $\tau$ and $N$ & $n=100$, $\tau=1+\frac{3}{2}\im$, $N=40$ & & \\ [0.5ex]
 exact numerical value of $C_{2n-1}(\tau)$ & $0.6104582432674722845198873$ & $\times$ & $10^{75}$ \\ [0.1ex]
 & $-\im\, 0.3947300517906354580061330$ & $\times$ & $10^{75}$ \\ [0.5ex]
 approximation \eqref{eq26} to $C_{2n-1}(\tau)$ & $0.6104582432674722845199135$ & $\times$ & $10^{75}$ \\ [0.1ex]
 & $-\im\, 0.3947300517906354580061436$ & $\times$ & $10^{75}$ \\ [0.5ex]
 error & $-0.262$ & $\times$ & $10^{53}$\\ [0.1ex]
& $+\im\, 0.106$ & $\times$ & $10^{53}$ \\ [-1.5ex]
 & \\\hline
\end{tabular}
\caption{Approximations for $C_{2n-1}(\tau)$ with various $n$, $\tau$ and $N$, using \eqref{eq26}.}
\label{table1}
\end{table*}

\begin{table*}[!ht]\centering
\begin{tabular}
[c]{ l r @{\,}c@{\,} l}\hline
 & \\ [-1.5ex]
 values of $n$, $\tau$ and $N$ & $n=50$, $\tau=1$, $N=10$ & & \\ [0.5ex]
 exact numerical value of $C_{2n}(\tau)\qquad$ & $0.9780541202848348054115227$ & $\times$ & $10^{23}$ \\ [0.5ex]
 approximation \eqref{eq27} to $C_{2n}(\tau)$ & $0.9780541202841343977234399$ & $\times$ & $10^{23}$ \\ [0.5ex]
 error & $0.7004076880828$ & $\times$ & $10^{11}$\\ [-1.25ex]
 & \\\hline
& \\ [-1.5ex]
 values of $n$, $\tau$ and $N$ & $n=50$, $\tau=1$, $N=20$ & & \\ [0.5ex]
 exact numerical value of $C_{2n}(\tau)$ & $0.9780541202848348054115227$ & $\times$ & $10^{23}$ \\ [0.5ex]
 approximation \eqref{eq27} to $C_{2n}(\tau)$ & $0.9780541202848340460370285$ & $\times$ & $10^{23}$ \\ [0.5ex]
 error & $0.7593744942$ & $\times$ & $10^{8}$\\ [-1.5ex]
 & \\\hline
& \\ [-1.5ex]
 values of $n$, $\tau$ and $N$ & $n=100$, $\tau=2+\im$, $N=20$ & & \\ [0.5ex]
 exact numerical value of $C_{2n}(\tau)$ & $0.3119948787485535986155779$ & $\times$ & $10^{77}$ \\ [0.1ex]
 & $-\im\, 0.6504249040471427943527241$ & $\times$ & $10^{77}$ \\ [0.5ex]
 approximation \eqref{eq27} to $C_{2n}(\tau)$ & $0.3119948786903918092779561$ & $\times$ & $10^{77}$ \\ [0.1ex]
 & $-\im\, 0.6504249040964131036369408$ & $\times$ & $10^{77}$ \\ [0.5ex]
 error & $0.581617893376218$ & $\times$ & $10^{67}$\\ [0.1ex]
& $-\im\, 0.492703092842167$ & $\times$ & $10^{67}$ \\ [-1.5ex]
 & \\\hline
& \\ [-1.5ex]
 values of $n$, $\tau$ and $N$ & $n=100$, $\tau=2+\im$, $N=40$ & & \\ [0.5ex]
 exact numerical value of $C_{2n}(\tau)$ & $0.3119948787485535986155779$ & $\times$ & $10^{77}$ \\ [0.1ex]
 & $-\im\, 0.6504249040471427943527241$ & $\times$ & $10^{77}$ \\ [0.5ex]
 approximation \eqref{eq27} to $C_{2n}(\tau)$ & $0.3119948786894443063970208$ & $\times$ & $10^{77}$ \\ [0.1ex]
 & $-\im\, 0.6504249040965108927273969$ & $\times$ & $10^{77}$ \\ [0.5ex]
 error & $0.59109292218557109$ & $\times$ & $10^{67}$\\ [0.1ex]
& $-\im\, 0.49368098374672863$ & $\times$ & $10^{67}$ \\ [-1.5ex]
 & \\\hline
\end{tabular}
\caption{Approximations for $C_{2n}(\tau)$ with various $n$, $\tau$ and $N$, using \eqref{eq27}.}\label{table2}
\end{table*}

\section{Concluding remarks}\label{section4}

We studied the asymptotic behaviour of the coefficients $C_n(\tau)$ as $n\to +\infty$ appearing in the asymptotic expansion of the incomplete gamma function within the transition region. We established inverse factorial series expansions for $C_n(\tau)$ that exhibit a resurgence property--meaning the coefficients of these series are once again $C_k(\tau)$. To the best of our knowledge, there has been no investigation into the resurgence properties of asymptotic expansions in transition regions in the existing literature prior to this paper.

The example of the incomplete gamma function is specific in a sense, as its derivative can be expressed in terms of the gamma function, whose resurgence properties are well understood. It would be interesting to develop a more general method for studying the resurgence properties of transitional asymptotic expansions that lack such special properties. Two important examples include the transitional expansions of the Bessel functions \cite[Section~10.19\,(iii)]{DLMF} and the Coulomb functions \cite[Section~33.12\,(i)]{DLMF}, both of which do not exhibit these special characteristics. We anticipate that the Borel transform will play a crucial role in such investigations. Additionally, exploring beyond the asymptotic study of coefficients and developing a hyperasymptotic theory for transitional expansions, similar to those existing for standard (non-uniform) asymptotic expansions \cite{Bennett2018,OldeDaalhuis1998b}, would be of great interest.

\appendix

\section{Auxiliary results}\label{appendixa}

In this appendix, we provide a concise overview of the resurgence properties of the asymptotic expansion of the gamma function and the uniform asymptotic expansion of the incomplete gamma function. Our focus is on the asymptotic behaviour of the coefficients of these expansions, as they play a significant role in proving our main result.

We define the scaled gamma function $\Gamma^\ast(a)$ through the relation
\begin{equation}\label{eq24}
\Gamma^\ast (a) = \frac{\Gamma (a)}{\sqrt {2\pi } a^{a - 1/2} \e^{ - a} }
\end{equation}
valid for $|\arg a|<\pi$. It is a well-established fact that both the scaled gamma function and its reciprocal possess asymptotic expansions given by
\begin{equation}\label{eq25}
\Gamma^\ast (a) \sim \sum_{n = 0}^\infty ( - 1)^n \frac{\gamma _n }{a^n } ,\qquad \frac{1}{\Gamma^\ast (a)} \sim \sum_{n = 0}^\infty \frac{\gamma _n }{a^n }
\end{equation}
as $a\to \infty$ in the sector $|\arg a| \le \pi-\delta<\pi$ (see, e.g., \cite[Section~3.6]{Temme1996}). Here, the $\gamma_n$ are known as the Stirling coefficients, with the initial values being
\[
\gamma _0 = 1,\qquad\gamma _1 = - \frac{1}{12},\qquad\gamma _2 = \frac{1}{288},\qquad\gamma _3 = \frac{139}{51840},\qquad\gamma _4 = - \frac{571}{2488320}.
\]
Note that the asymptotic expansion of the reciprocal scaled gamma function involves the same coefficients as that of the scaled gamma function, but with different signs of the coefficients with odd index. For an explanation of this phenomenon, refer to, for instance, \cite[p.~63]{Temme1996}. The asymptotic behaviour of the Stirling coefficients is described by the inverse factorial series
\begin{equation}\label{eq18}
\gamma _n \sim -\frac{1}{\pi}\sum_{k= 0}^{\infty} \gamma_k \sin\left( \frac{n -k}{2}\pi \right)\frac{\Gamma (n - k)}{(2\pi )^{n - k}}
\end{equation}
as $n\to+\infty$. Notably, the coefficients within the inverse factorial series are once again the Stirling coefficients, highlighting the resurgence property of the gamma function. The expansion \eqref{eq18} was established by Boyd \cite[equation~(3.34)]{Boyd1994}, although it had previously appeared without proof in an earlier paper by Rosser \cite[equation~(97)]{Rosser1955}.

Moving on to the incomplete gamma function $Q(a,z)$, define
\[
\lambda = \frac{z}{a},\qquad \eta = (2(\lambda - 1 - \log \lambda ))^{\frac{1}{2}} ,
\]
where the branch of the square root is continuous and satisfies $\eta(\lambda)\sim \lambda-1$ as $\lambda\to 1$. Temme~\cite{Temme1979} established that the incomplete gamma function admits the uniform asymptotic expansion
\begin{equation}\label{eq23}
Q(a,z) \sim \frac{1}{2}\erfc\bigg( \eta \sqrt{\frac{a}{2}} \bigg) + \frac{1}{\sqrt{2\pi a} }\exp \left(- \eta ^2 \frac{a}{2}\right)\sum_{n = 0}^\infty \frac{c_n (\eta )}{a^n},
\end{equation}
as $a\to \infty$ in the sector $|\arg a| \le \pi-\delta<\pi$, uniformly with respect to $\lambda$ in the sector $|\arg \lambda| \le 2\pi-\delta<2\pi$. The coefficients $c_n (\eta )$ are complicated functions of their argument with removable singularities at $\eta=0$. The resurgence properties of this expansion were studied by Olde Daalhuis~\cite{OldeDaalhuis1998a}. He demonstrated that the coefficients $c_n (\eta)$ possess inverse factorial series of the following form as $n\to+\infty$, for $\eta \in \mathcal{N}$:
\begin{gather}\label{eq22}
\begin{split}
c_n (\eta ) \sim \; & \frac{( - 1)^{n + 1} }{2\sqrt {2\pi }}\Gamma \left( n + \frac{1}{2} \right)\left( \frac{1}{2}\eta ^2 + 2\pi \im \right)^{ - n - \frac{1}{2}} I_{1 + \frac{\eta ^2 }{4\pi \im}} \left( n + \frac{1}{2},\frac{1}{2} \right)
\\ & + \frac{( - 1)^{n + 1} }{2\sqrt {2\pi } }\Gamma \left( n + \frac{1}{2} \right)\left( \frac{1}{2}\eta ^2 - 2\pi \im \right)^{ - n - \frac{1}{2}} I_{1 - \frac{\eta ^2 }{4\pi \im}} \left( n + \frac{1}{2},\frac{1}{2} \right)
\\ & - \frac{1}{\pi}\sum_{k = 0}^\infty c_k (\eta )\sin \left( \frac{n - k}{2}\pi \right)\frac{\Gamma (n - k)}{(2\pi )^{n - k} },
\end{split}
\end{gather}
where
\[
\mathcal{N} = \left\{ \eta \in \mathbb{C} \right\}\setminus \big\{ \eta = \alpha + \im \beta \mid \alpha \beta = \pm 2\pi ,\; \alpha \le - \sqrt {2\pi}\big\}
\]
is the image of the sector $|\arg \lambda|<2\pi$ under the mapping $\eta(\lambda)$. In \eqref{eq22}, $I_x(a,b)$ denotes the incomplete beta function \cite[Section~8.17]{DLMF}. In the special case that $\eta=0$, the expansion \eqref{eq22} simplifies to the inverse factorial series
\begin{equation}\label{eq19}
c_n (0) \sim -\frac{\Gamma \left( n + \frac{1}{2} \right)}{(2\pi )^{n + 1}}\sin\left( \left( n + \frac{1}{2} \right)\frac{\pi}{2} \right) - \frac{1}{\pi}\sum_{k = 0}^\infty c_k (0)\sin \left( \frac{n - k}{2}\pi \right)\frac{\Gamma (n - k)}{(2\pi )^{n - k} }
\end{equation}
as $n\to+\infty$.

An alternative asymptotic expansion for $c_n (\eta)$ was established in the paper \cite{Dunster1998}, but it does not exhibit the resurgence property.

\section[The computation of the coefficients C\_n(tau)]{The computation of the coefficients $\boldsymbol{C_n(\tau)}$}\label{appendixb}

In this appendix, we present an efficient method for computing the coefficients $C_n(\tau)$. We~substitute the polynomial expansion
\begin{equation}\label{eq29}
C_n(\tau) = \sum_{k=0}^{3n+2} c_{n,k}\tau^k
\end{equation}
into \eqref{eq32}. As a result, this substitution yields the following recurrence relation for the coefficients~$c_{n,k}$:
\begin{equation}\label{eq30}
c_{n,k}=(k+2)c_{n,k+2}+(k+1) c_{n-1,k+1}-\frac{2k}{k+1}c_{n-1,k-1}+\frac{1}{k+1}c_{n-1,k-3}.
\end{equation}
Given that $C_0 (\tau ) = \frac{1}{3}\tau^2 -\frac{1}{3}$, we deduce that
\begin{equation}\label{eq31}
c_{n,3n+2}=\frac{1}{3^{n+1}(n+1)!},\qquad c_{n,3n+1}=0
\end{equation}
holds for $n=0$. Using the recurrence relation \eqref{eq30}, we can demonstrate that \eqref{eq31} remains valid for $n\geq1$. Initiating with the first equation in \eqref{eq31} and subsequently setting $k=3n, 3n-2, \allowbreak 3n-4, \dots$, enables the computation of non-zero coefficients in the expansion \eqref{eq29}. Moreover, we establish that $c_{n,k}=0$ for $k=3n+1, 3n-1, 3n-3, \dots$. The following Wolfram Mathematica~\cite{Mathematica} code implements the recursive method:
\[
\begin{aligned}
&\bm{c[0,2] = 1/3;} \\
&\bm{c[0,0] = -1/3;} \\
&\bm{c[0,4] = 0;} \\
&\bm{c[0,-2] = 0;} \\
&\bm{\textbf{NN} = 10;} \\
&\bm{\textbf{For}[n = 1, n \leq \textbf{NN}, n\textbf{++},} \\
&\quad \bm{c[n,3*n+2] = 1/3{}^{\wedge}(n+1)/(n+1)!;} \\
&\quad \bm{c[n,3*n+4] = 0;} \\
&\quad \bm{\textbf{For}[k = 3*n, k \geq 0, k-=2,} \\
&\quad \quad \bm{c[n,k] = (k+2)*c[n,k+2] + (1/(k+1))*c[n-1,k-3]} \\
&\quad \quad \quad \bm{- (2k/(k+1))*c[n-1,k-1] + (k+1)*c[n-1,k+1];];} \\
&\quad \bm{c[n,k] = 0;} \\
&\quad \bm{c[n,k-2] = 0;];} \\
&\bm{\textbf{Ccoefficient}[\textbf{n$_{-}$},\textbf{t$_{-}$}]\,\textbf{:=}\,\textbf{Sum}[c[n,3*n+2-2*k]*t{}^{\wedge}(3*n+2-2*k),}\\
& \hphantom{\textbf{Ccoefficient}[\textbf{n$_{-}$},\textbf{t$_{-}$}]\,\textbf{:=}\,} \bm{\{k,0,3/2*n+1\}];} \\
&\bm{\textbf{Table}[\textbf{Ccoefficient}[n,t],\{n,0,\textbf{NN}\}]}
\end{aligned}
\]
This code generates a list of coefficients $C_n(t)$ for $0\le n\le 10$. Adjusting the value of the variable \textbf{NN} allows for computing a different number of coefficients.

\subsection*{Acknowledgements} The author's research was supported by the JSPS KAKENHI Grants No.~JP21F21020 and No.~22H01146. The author is grateful to the referees for their comprehensive evaluation of the manuscript.

\pdfbookmark[1]{References}{ref}
\LastPageEnding

\end{document}